\documentclass{article}
\usepackage[margin= 22mm,bottom=30mm, top= 21mm]{geometry}
\usepackage{amsthm}
\usepackage{amsmath}
\usepackage{amssymb}
\usepackage{setspace}
\usepackage{mathtools}
\usepackage{verbatim}
\usepackage{a4wide}
\usepackage{csquotes}
\usepackage[hidelinks]{hyperref}
\usepackage{cleveref}
\usepackage{enumitem}
\setlist[itemize]{leftmargin=*} 
\setlist[enumerate]{leftmargin=*}
\usepackage{framed}
\usepackage{floatrow}
\usepackage[T1]{fontenc}
\usepackage{bbm}

\usepackage{mathdots}
\usepackage{xcolor}
\usepackage{diagbox}
\usepackage{colortbl}

\usepackage{graphicx}
\usepackage{subcaption}

\theoremstyle{plain}

\newtheorem{theorem}{Theorem}[section]
\newtheorem{claim}[theorem]{Claim}

\newtheorem{lemma}[theorem]{Lemma}

\theoremstyle{definition}

\newtheorem*{defn*}{Definition}
\newtheorem*{lem*}{Lemma}
\newtheorem*{claim*}{Claim}

\def\domagoj#1{{}{\textcolor{blue}{#1} }}

\expandafter\def\expandafter\normalsize\expandafter{%
    \normalsize
    \setlength\abovedisplayskip{4pt}
    \setlength\belowdisplayskip{4pt}
    \setlength\abovedisplayshortskip{4pt}
    \setlength\belowdisplayshortskip{4pt}
}

\usepackage[square,sort,comma,numbers]{natbib}
\def\domagoj#1{}

\let\domagoj=\domagojOpt 

\def\eps {\varepsilon}
\DeclareMathOperator{\E}{\mathbb{E}}

\newcommand{\cT}{\mathcal{T}}

\newcommand{\bF}{\mathbb{F}}

\newcommand{\ba}{\mathbf{a}}
\newcommand{\bb}{\mathbf{b}}

\newcommand{\br}{\mathbf{r}}
\newcommand{\bx}{\mathbf{x}}
\newcommand{\by}{\mathbf{y}}
\newcommand{\bz}{\mathbf{z}}

\renewcommand{\Pr}{\mathbb{P}}
\renewcommand{\E}{\mathbb{E}}

\newcommand{\fwi}[1]{\overrightarrow{i_{#1}}}

\makeatletter
\newcommand{\bigplus}{%
  \DOTSB\mathop{\mathpalette\mattos@bigplus\relax}\slimits@
}
\newcommand\mattos@bigplus[2]{%
  \vcenter{\hbox{%
    \sbox\z@{$#1\sum$}%
    \resizebox{!}{0.9\dimexpr\ht\z@+\dp\z@}{\raisebox{\depth}{$\m@th#1+$}}%
  }}%
  \vphantom{\sum}%
}
\makeatother

\title{Off-diagonal Ramsey numbers}
\author{Domagoj Brada\v{c} \thanks{Institute of Mathematics, EPFL, Lausanne, Switzerland. Email: \textbf{domagoj.bradac@epfl.ch}}}
\date{}
\begin{document}

\maketitle

\begin{abstract}
    For positive integers $s$ and $k$, the Ramsey number $r(s,k)$ is the minimum integer $n$ such that any graph on $n$ vertices contains a clique of size $s$ or an independent set of size $k$. We prove that for any fixed $s \ge 3$ and $k$ tending to infinity, the off-diagonal Ramsey numbers satisfy 
    \[ r(s, k) \ge \Omega \left(\frac{k^{s-1}}{(\log k)^{2s-4}} \right), \] which matches, up to polylogarithmic factors, the upper bound established over 90 years ago by Erd\H{o}s and Szekeres. For $s \ge 5,$ this improves the best known lower bound of the form $r(s, k) \ge k^{\frac{s+1}{2} + o(1)}$ which was first established by Spencer in 1977 and has since only seen polylogarithmic improvements.
\end{abstract}

\section{Introduction}
Ramsey theory is a branch of mathematics guided by the philosophy that any large structure, no matter how disordered, contains a fairly large ordered substructure. The field bears its name after the logician F. P. Ramsey~\cite{ramsey}, but Ramsey theoretic statements and techniques arise in many other fields including set theory, topology, discrete geometry, theoretical computer science, ergodic theory, additive number theory and, of course, the larger field of combinatorics. Notable examples include van der Waerden's theorem on arithmetic progressions in the integers~\cite{vanderwaerden}, Schur's theorem stating that Fermat's last theorem is false modulo primes~\cite{schur}, Rado's partition regularity for systems of linear equations~\cite{rado-partition}, or the more recent Gowers' infinite Ramsey theorem for Banach spaces~\cite{gowers-banach}.  We refer to the book by Graham, Rothschild and Spencer~\cite{graham-rothschild-spencer} for an introduction to the area.

For positive integers $s$ and $k$, the \emph{Ramsey number}, $r(s, k)$ is the minimum integer $n$ such that any graph on $n$ vertices contains a clique of size $s$, which we denote by $K_s$, or an independent set of size $k$. The fact that these numbers are finite is a special case of the seminal theorem of Ramsey~\cite{ramsey} from 1929 and since then, a great deal of effort has been expended toward estimating these numbers as well as various related quantities -- see the survey by Conlon, Fox and Sudakov~\cite{CFSsurvey} for a broad overview of the many related questions. The two most natural regimes to consider are when $s = k,$ the so-called \emph{diagonal Ramsey numbers}, and when $s$ is fixed and $k \rightarrow \infty$, the so-called \emph{off-diagonal Ramsey numbers}. For a detailed history of the diagonal Ramsey numbers, we refer to the recent breakthrough work of Campos, Griffiths, R. Morris and Sahasrabudhe~\cite{CGMS} and to the follow-up works of Gupta, Ndiaye, Norin and Wei~\cite{gupta} for quantitative improvements and of Balister, Bollob\'{a}s, Campos, Griffiths, Hurley, R. Morris, Sahasrabudhe and Tiba~\cite{balister} for a multicolor generalization. The main focus of this paper are the off-diagonal Ramsey numbers.

The first reasonable quantitative upper bounds on the Ramsey numbers were given by Erd\H{o}s and Szekeres~\cite{erdos-szekeres} in 1935, who showed that for any positive $s$ and $k$, 
\begin{equation} \label{eq:erdos-szekeres}
    r(s, k) \le \binom{k+s-2}{s-1}.
\end{equation}

For the off-diagonal Ramsey numbers, \eqref{eq:erdos-szekeres} implies that $r(s, k) = O(k^{s-1})$, for any fixed $s$. This was improved by Ajtai, Koml\'{o}s and Szemer\'{e}di~\cite{ajtai} who showed that $r(s, k) = O(k^{s-1} / (\log k)^{s-2})$, and the leading constant was improved to $(1+o(1))$ by Li, Rousseau, and Zang~\cite{li-rousseau-zang}, generalizing the work of Shearer~\cite{shearer} who obtained such an improvement for $s=3$.

Using the so-called Lov\'{a}sz local lemma, Spencer~\cite{spencer75, spencer77} proved the first nontrivial lower bound for general $s$, showing that $r(s, k) = \Omega\left((k / \log k)^{\frac{s+1}{2}}\right).$ The polylogarithmic factor was slightly improved by Bohman and Keevash~\cite{bohman-keevash-Ks} by analyzing the random $K_s$-free process, the result of which is that for fixed $s \ge 5$, the best known bounds on $r(s, k)$ are

\begin{equation} \label{eq:old-bounds}
    \Omega\left(\frac{k^{\frac{s+1}{2}}}{(\log k)^{\frac{s+1}{2} - \frac{1}{s-2}}} \right) \le r(s, k) \le (1 + o(1))\frac{k^{s-1}}{(\log k)^{s-2}}.
\end{equation} 

The exponent in the above lower bound corresponds to the so-called deletion threshold for $K_s$ and it represents a natural barrier unlikely to be surpassed by any purely probabilistic construction\footnote{Although, for the so-called odd cycle-complete Ramsey numbers, the deletion threshold has recently been broken by Campos, Jenssen, Michelen, Pfender and Sahasrabudhe~\cite{oddcyclecomplete} by superimposing blow-ups of random graphs.}. The conjecture that for any fixed $s$, $r(s, k) \ge \frac{k^{s-1}}{(\log k)^c},$ for some $c = c(s)$, appears to have been made by Erd\H{o}s in 1947~\cite{chung-graham} and several authors~\cite{spencer77, mubayi-verstraete, CMMV-zarankiewicz} have tepidly supported the conjecture that at least $r(s, k) \ge k^{s-1+o(1)}$ should hold. Our main result proves this conjecture and thus determines off-diagonal Ramsey numbers up to polylogarithmic factors.

\begin{theorem} \label{thm:main}
    For any $s \ge 3$, there is a positive constant $c_s$ such that for any $k \ge 2,$
    \[ r(s, k) \ge c_s \frac{k^{s-1}}{(\log k)^{2s-4}}. \]
\end{theorem}
In the first version of this work, the author proved the weaker lower bound $r(s, k) \ge \Omega\left( \frac{k^{s-2}}{(\log k)^{2s-6}} \right)$. The final improvement to obtain Theorem~\ref{thm:main} was made by an internal model at OpenAI based on those ideas. We remark that Theorem~\ref{thm:main} improves the previously best known lower bounds for all $s \ge 5$. We shall shortly discuss the history of the  important special cases $s = 3$ and $s = 4$, focusing on the ideas that inspired the present approach. First, we mention several other applications, all of which follow from the author's original construction.

For general values of $s$ and $k$, in an excellent survey of recent breakthroughs in the area, R. Morris~\cite{morris-survey} showed that combining three proofs suitable for different regimes, the upper bound in~\eqref{eq:erdos-szekeres} can always be improved by an exponential in $s$, i.e. there is an absolute constant $\delta > 0$ such that for all sufficiently large $k$ and all $s \le k,$
\begin{equation} \label{eq:morris}
    r(s, k) \le e^{-\delta s} \binom{k+s-2}{s-1}.
\end{equation}

On the other hand, the first moment lower bound of Erd\H{o}s~\cite{erdos-lb} as well as the local lemma approach of Spencer can easily be adapted to give lower bounds for general $s$ and $k$. When $k = \Theta(s)$, the local lemma only provides a constant factor improvement over the first moment method argument. In a recent breakthrough, Ma, Shen and Xie~\cite{ma-shen-xie} have obtained an exponential improvement on this lower bound.

\begin{theorem}[Ma, Shen, Xie~\cite{ma-shen-xie}] \label{thm:ma-shen-xie}
    For any constant $C > 1$, there exists a constant $\eps = \eps(C) > 0$ such that for any sufficiently large $s$,
    \[ r(s, Cs) \ge (p_C^{-1/2} + \eps)^s, \]
    where $p_C \in (0, 1/2)$ is the unique solution to $C  = \frac{\log p_C}{\log (1 - p_C)}$.
\end{theorem}

For context, we remark that the first moment method achieves the lower bound $r(s, Cs) \ge (p_C^{-1/2} + o(1))^s$. To prove this lower bound, Ma, Shen and Xie use a random geometric graph. Their result was reproved by Hunter, Milojevi\'{c} and Sudakov~\cite{hunter-milojevic-sudakov} using a similar construction, based on Gaussian random graphs, which allows for a simpler analysis and achieves better quantitative dependencies when $C \rightarrow \infty$; see also~\cite{lin-niu} for a sharper analysis leading to a small improvement in the final bounds.

If $s, k/s \rightarrow \infty$, the lower bounds given by the first moment method, the local lemma and the recent improvements in~\cite{ma-shen-xie},~\cite{hunter-milojevic-sudakov}~and~\cite{lin-niu} all yield a lower bound of the form
\begin{equation} \label{eq:spencer-general}
    r(s, k) \ge \left( \frac{k}{s} \right)^{(1 + o(1))s/2}.
\end{equation}
Therefore, if $s, k/s \rightarrow \infty$, the lower bound \eqref{eq:spencer-general} and the upper bound \eqref{eq:erdos-szekeres} (or the improved~\eqref{eq:morris}) imply that\linebreak $(1 / 2 + o(1)) s \log(k/s) \le \log r(s, k) \le (1 + o(1)) s \log(k/s)$. We close this gap by showing that the upper bound is asymptotically tight.

\begin{theorem} \label{thm:off-diagonal-general}
    For every $\delta > 0$, there exists a constant $L$ such that for any positive integers $s \ge L$ and $k \ge Ls$, it holds that 
    \[ r(s, k) \ge \left( \frac{k}{s}\right)^{(1 - \delta) s}. \]
\end{theorem}

Note that Theorem~\ref{thm:off-diagonal-general} improves Theorem~\ref{thm:ma-shen-xie} for any sufficiently large $C$. We are also able to improve these bounds for $C$ close to $1$. More formally, for some positive $\gamma > 0$, we obtain an exponential improvement for any $C \in (1, 1+\gamma)$.
\begin{theorem} \label{thm:k-Ck}
    Let $C > 1$ be a fixed constant. Then, for large enough $s$,
    \[ r(s, Cs) \ge (2^{1 - \frac{1}{2C}})^s. \]
\end{theorem}
Writing $C = 1 + \alpha$ and considering $\alpha \searrow 0$, Theorem~\ref{thm:k-Ck} yields $r(s, Cs) \ge (p_C^{-1/2} + \eps)^s$, with $\eps = \Omega(\alpha)$, whereas~\cite{ma-shen-xie},~\cite{hunter-milojevic-sudakov}~and~\cite{lin-niu} all achieve the weaker $\eps = \Omega(\alpha^2)$.

Moving even closer to the diagonal, for $a = o(s)$, a simple, though slightly tedious, application of the local lemma, the optimality of which we demonstrate in the appendix, shows that 
\begin{equation} \label{eq:spencer-close-to-diag}
    r(s, s+a) \ge (1 + o(1)) \frac{s}{e} \cdot 2^{(s+a/2+1)/2 + O(a^2/s)}.
\end{equation}

Ma, Shen and Xie~\cite{ma-shen-xie} obtain an improvement over~\eqref{eq:spencer-close-to-diag} for $a = \omega(\sqrt{s})$. We are able to improve these lower bounds as soon as $a \ge 5$.
\begin{theorem} \label{thm:close}
    Let $s \rightarrow \infty$ and let $a$ be a nonnegative integer such that $a = o(s)$. Then, 
    \[ r(s, s+a) \ge (1 + o(1)) \frac{s}{e}\cdot 2^{(s+a-1)/2 - a^2 / (2s)}. \]
\end{theorem}

For positive integers $s$ and $\ell$, the multicolor Ramsey number $r(s; \ell)$ is the minimum integer $n$ such that any \mbox{$\ell$-coloring} of the edges of $K_n$ contains a monochromatic clique of size $s$. Using a construction based on polarity graphs, which also play a central role in the present paper, Conlon and Ferber~\cite{conlon-ferber} were first to obtain an exponential improvement over the old lower bound obtained using a random graph and a simple product argument of Abbott~\cite{abbott}. Their bound has subsequently been improved by Wigderson~\cite{wigderson}, Sawin~\cite{sawin}, and more recently by Campos and Pohoata~\cite{campos-pohoata}, and Attwa, Vidal and P. Morris~\cite{attwa-vidal-morris}, independently. The current best lower bounds on $r(s; \ell)$, given by~\cite{campos-pohoata} and~\cite{attwa-vidal-morris}, are less than $2^{0.4(\ell-2) s + s/2}$. We obtain the following improvement.

\begin{theorem} \label{thm:multicolor}
    For any fixed $\ell \ge 3$,
    \[ r(s; \ell) = \Omega(2^{(\ell-1) s / 2}). \]
\end{theorem}
While this presents an exponential improvement for any fixed $\ell$, for large $\ell$ it is still far from the best known upper bound $r(s; \ell) \le e^{-\delta s} \ell^{\ell s}$, where $\delta = \delta(\ell) > 0$, obtained by Balister, Bollob\'{a}s, Campos, Griffiths, Hurley, R. Morris, Sahasrabudhe and Tiba~\cite{balister}.

Our construction draws inspiration from several sources including some of the recent works on $r(3, k)$ and $r(4, k)$. The starting point of our proof stems from ideas based on pseudorandom graphs. A particularly convenient notion of pseudorandomness is that of so-called $(n, d, \lambda)$-graphs which are $d$-regular, $n$-vertex graphs whose adjacency matrices have all but the largest eigenvalue at most $\lambda$ in absolute value. This notion has first been systematically used by Alon and there has been a great deal of work utilising these graphs and studying their properties. We refer the interested reader to the excellent survey of Krivelevich and Sudakov~\cite{pseudorandomgraphssurvey} for an overview of the many results regarding pseudorandom graphs.

Alon and R\"{o}dl~\cite{alon-rodl} used $(n,d,\lambda)$-graphs to lower bound certain multicolor Ramsey numbers. These ideas were built upon by Mubayi and Verstra\"{e}te~\cite{mubayi-verstraete} who proved that the existence of dense pseudorandom $K_s$-free graphs implies lower bounds for the off-diagonal Ramsey numbers $r(s, k)$. In particular, if there exists $K_s$-free $(n, d, \lambda)$-graphs with $d = \Omega(n^{1 - 1 / (2s-3)})$ and $\lambda = O(\sqrt{d})$, which would be optimal, then $r(s, k) \ge k^{s-1 + o(1)}$. The idea is to show that randomly sampling vertices of a dense pseudorandom $K_s$-free graph with high probability produces a graph without large independent sets which in turn, translates to a lower bound on $r(s, k)$. Due to our limited knowledge of dense pseudorandom $K_s$-free graphs, which we shall discuss below, this unfortunately has not directly produced improved lower bounds for off-diagonal Ramsey numbers. Nonetheless, it gave improvements to a related problem of so called odd cycle-complete Ramsey numbers~\cite{mubayi-verstraete, CMMV-zarankiewicz}, and more importantly, it served as an inspiration for the recent resolution of the exponent of $r(4, k)$, which we discuss next.

In a major breakthrough, Mattheus and Verstraete~\cite{mattheus-verstraete} proved that $r(4,k) = \Omega(k^3 / (\log k)^4)$ which, by~\eqref{eq:old-bounds}, is tight up to a factor of $O(\log k)^2$. Their construction uses so-called Hermitian unitals from finite geometry to produce an algebraically defined graph which can be randomly modified to obtain a $K_4$-free graph with a nice edge distribution. While this graph is not spectrally pseudorandom, its independent sets can be efficiently counted using the so-called container method; see~\cite{samotij-survey} for an introduction to this powerful technique. Having a good upper bound on the number of large independent sets, as in~\cite{mubayi-verstraete}, taking a random subset of vertices yields the lower bound on $r(4, k)$. This proof strategy has been codified by Conlon, Mattheus, Mubayi and Verstra\"{e}te~\cite{CMMV-zarankiewicz}, the upshot of which is that if some ``natural conjectures'' on a certain extremal problem are true, then $r(s, k) = k^{s-1 + o(1)}$. An outline of the proof of the lower bound on $r(4, k)$ as well as a survey of recent results utilizing the pseudorandom graph framework was recently presented by Verstr\"{a}ete~\cite{verstraete-survey}.

As we shall see, our approach has a similar flavour -- we start with an algebraically defined construction, which can be thought of, and in fact easily transformed into, a $K_s$-free graph. This graph is not spectrally pseudorandom but its independent sets can be counted efficiently using the container method, closely following the argument of Alon and R\"{o}dl~\cite{alon-rodl}. Again, taking a random subset of vertices yields our lower bound on $r(s, k)$.

Next, we briefly discuss the case $s=3$, while for a much more detailed history we refer to~\cite{CJMSk3}. As mentioned, the best known upper bound on $r(3, k)$ is due to Shearer~\cite{shearer} and is the right hand side of~\eqref{eq:old-bounds}. The first lower bound determining the order of magnitude of $r(3, k)$ is due to Kim~\cite{kim}. The leading constant has been improved by Fiz Pontiveros, Griffiths and R. Morris~\cite{fizpontiveros} and, independently, by Bohman and Keevash~\cite{bohman-keevash-K3} using the triangle-free process. This lower bound was recently improved by Campos, Jenssen, Michelen and Sahasrabudhe~\cite{CJMSk3}. Drawing some high level ideas from their work, Hefty, Horn, King and Pfender~\cite{heftyhorn} found a beautiful construction that narrows the gap between the upper and lower bound to a factor of only $2 + o(1)$. The construction in~\cite{heftyhorn} is based on a certain graph product partly inspiring the present construction, which, as we shall see, has a product-like flavor.

Let us address the final ingredient in our proof. Our construction is based on the polarity graphs of projective spaces. These were shown by Alon and Krivelevich~\cite{alon-krivelevich} to provide instances of $K_s$-free $(n, d, \lambda)$-graphs with $d = \Theta(n^{1 - 1/(s-2)})$ and $\lambda = O(\sqrt{d})$. It is worth mentioning that Bishnoi, Ihringer and Pepe~\cite{bishnoiihringer} (see also Mattheus and Pavese~\cite{mattheus-pavese}) used a particular polarity to construct $K_s$-free graphs with a larger density, namely $d = \Theta(n^{1 - 1 / (s-1)})$ and also optimal spectral pseudorandomness. However, we do not actually use that the polarity graphs are $K_s$-free, but rather that they avoid a certain bipartite configuration.

\textbf{Declaration on the use of AI.} As mentioned, the argument that improved the exponent from $s-2$ to $s-1$ in Theorem~\ref{thm:main} was found by an internal model at OpenAI and communicated to the author. All the remaining ideas and proofs were derived with no significant use of AI tools. The computation in the appendix was derived using Claude. The paper is written entirely by the author.

\textbf{Organisation.} The rest of this paper is structured as follows. Section~\ref{sec:main} contains the proof of our main results and it is split into five subsections. First, in Subsection~\ref{subsec:Gtq} we introduce the polarity graphs from~\cite{alon-krivelevich} which serve as the basis of our construction and (re)prove the properties we need. In Subsection~\ref{subsec:sketch}, we sketch a proof of the weaker bound $r(s, k) \ge k^{s-2 + o(1)}$ given this construction. Subsection~\ref{subsec:prelims} introduces a bit of notation and contains several preliminary lemmas. In Subsection~\ref{subsec:old}, we present a proof of the bound $r(s, k) \ge k^{s-2 + o(1)}$ as well as a proof of Theorem~\ref{thm:off-diagonal-general}. In Subsection~\ref{subsec:s-1}, we prove Theorem~\ref{thm:main}. We present the proofs of Theorems~\ref{thm:k-Ck},~\ref{thm:close}~and~\ref{thm:multicolor} in Section~\ref{sec:close}. Finally, we end with some concluding remarks in Section~\ref{sec:concluding}.

\textbf{Notation.} We use mostly standard graph theoretic notation. Some of the graphs we consider have loops, at most one per vertex. In the adjacency matrix of a graph, a loop is represented by a $1$ as the corresponding diagonal entry and, accordingly, each loop contributes one to the degree of the vertex incident to it. For a graph $G$ and a vertex $v$ of $G$, we write $N_G(v)$ for the set of neighbours of $v$ in $G$, where $v \in N_G(v)$ if $G$ has a loop at $v$. For two vertex sets $A$ and $B$, we write $e_G(A, B) = \{ (a, b) \mid a \in A, b \in B, ab \in E(G) \}.$ For a directed graph, or digraph for short, $D$, we use $A(D)$ to denote its set of arcs. We use $\log x$ to denote the natural logarithm of $x$. We omit the use of floor and ceiling signs whenever they are not crucial to the argument.

\section{Setup and proofs for off-diagonal Ramsey numbers} \label{sec:main}

\subsection{The polarity graphs} \label{subsec:Gtq}
The basis of our construction is a construction of optimally pseudorandom $K_s$-free graphs due to Alon and Krivelevich~\cite{alon-krivelevich} which we now describe. Let $t \ge 2$ and let $q$ be an arbitrary prime power. Let $PG(t, q)$ denote the finite geometry of dimension $t$ over the field $\bF_q$. The points of $PG(t, q)$ may be represented by the $1$-dimensional subspaces, or $1$-spaces for short, of $\bF_q^{t+1}$, or equivalently, by equivalence classes of nonzero vectors $(x_1, \dots, x_{t+1})$ of length $t+1$ over $\bF_q$ under the equivalence relation where two vectors are equivalent if one is a multiple of the other by a nonzero element of $\bF_q$. Let $G(t, q)$ be the graph with vertex set $PG(t, q)$ where two (not necessarily distinct) vertices represented by vectors $\bx$ and $\by$ are joined by an edge if and only if $\langle \bx, \by \rangle = \sum_{i=1}^{t+1} x_i y_i = 0$, i.e. if the corresponding $1$-spaces are orthogonal. Note that whether $\langle \bx, \by \rangle = 0$ does not depend on the choice of the representatives, so the graph is well defined. Let us point out that $G(t, q)$ contains a loop at a vertex represented by $\bx$ if and only if $\langle \bx, \bx \rangle = 0$.

A graph $G$, which might contain loops, is said to be an $(n, d, \lambda)$-graph if it has $n$ vertices, it is $d$-regular, and all but the largest eigenvalue (which equals $d$) of its adjacency matrix are at most $\lambda$ in absolute value.

The following properties of $G(t, q)$ are shown in~\cite{alon-krivelevich} and we include a proof for completeness. The proof in~\cite{alon-krivelevich} assumes that $t$ is fixed and $q$ tends to infinity, while we shall only assume the latter. However, allowing $t$ to grow with $q$ makes no difference to the proof.

\begin{lemma}[\cite{alon-krivelevich}] \label{lem:properties}
    For $t \ge 2$, the graph $G(t, q)$ is an $(n, d, \lambda)$-graph, with $n = q^t (1 + o(1))$, $d = q^{t-1}(1 + o(1))$ and $\lambda = (1+o(1)) \sqrt{d}$, where the asymptotic notation is for $q$ tending to infinity.
\end{lemma}
\begin{proof}
    Each equivalence class corresponding to a point of $PG(t, q)$ has size $q-1$, so $G(t,q)$ has $n = (q^{t+1} - 1) / (q-1) = (1+o(1)) q^t$ vertices. For any nonzero $\bx \in \bF_q^{t+1},$ there are $q^t - 1$ nonzero solutions $\by$ to the equation $\langle \bx, \by \rangle = 0$, so $G(t,q)$ is $d$-regular with $d = (q^t - 1) / (q-1) = (1+o(1))q^{t-1}.$

    Let $\bx, \by \in \bF_q^{t+1}$ be linearly independent. Then, the number of nonzero vectors $\bz \in \bF_q^{t+1}$ such that $\langle \bx, \bz \rangle = \langle \by, \bz \rangle = 0$ equals $q^{t-1} - 1$, thus any pair of distinct vertices of $G(t,q)$ have $a \coloneqq (q^{t-1} - 1) / (q-1)$ common neighbours. Let $A$ denote the adjacency matrix of $G(t, q)$. Denoting by $J_n$ the $n \times n$ all-ones matrix and by $I_n$ the $n \times n$ identity matrix, the above facts imply that $A^2 = aJ_n + (d-a)I_n.$ Recalling that the principal eigenvector of the adjacency matrix of a $d$-regular graph is the all-ones vector and the other eigenvectors are orthogonal to it, it follows that the absolute value of all but the largest eigenvalue of $G(t,q)$ equals $\sqrt{d-a} = (1+o(1)) \sqrt{d}.$
\end{proof}

It was shown in~\cite{alon-krivelevich} that the subgraph of $G(t, q)$ induced on the set of non-self-orthogonal points is $K_{t+2}$-free. For our purposes, however, we need that $G(t, q)$ avoids a certain bipartite structure as proved in the lemma below. The proof is the skew version of the dimension argument exhibiting $K_{t+2}$-freeness.

\begin{lemma} \label{lem:Hs-free}
    There do not exist vertices $a_1, b_1, \dots, a_{t+2}, b_{t+2} \in V(G(t, q))$ such that $a_ib_j \in E(G(t, q))$ for all $i, j$ with $1 \le i < j \le t+2$ and $a_i b_i \not\in E(G(t, q))$ for all $i \in [t+2]$.
\end{lemma}
\begin{proof}
    Assume this is false, so there exist vectors $\bx_1, \by_1, \bx_2, \by_2, \dots, \bx_{t+2}, \by_{t+2} \in \bF_q^{t+1}$ such that $\langle \bx_i, \by_j \rangle = 0$ for all $i, j$ with $1 \le i < j \le t+2$ and $\langle \bx_i, \by_i \rangle \neq 0$ for all $i \in [t+2]$.
    
    We claim that the vectors $\by_1, \by_2, \dots, \by_{t+2}$ are linearly independent. Indeed, suppose there is a non-trivial linear combination $\sum_{j=1}^{t+2} \alpha_j \by_j = 0$ and let $i$ be the minimum index $j \in [t+2]$ such that $\alpha_j \neq 0$. We have
    \[ 0 = \sum_{j=1}^{t+2} \alpha_j \by_j = \sum_{j=i}^{t+2} \alpha_j \by_j. \]

    Taking the inner product with $\bx_i$ and recalling that $\langle \bx_i, \by_j \rangle = 0$ whenever $i < j \le t+2,$ it follows that
    \[ \alpha_i \langle \bx_i, \by_i \rangle = 0. \]
    Since $\langle \bx_i, \by_i \rangle \neq 0,$ it follows that $\alpha_i = 0$, contradicting our choice of $i$ and thus proving that the vectors $\by_1, \dots, \by_{t+2}$ are linearly independent. Since these are vectors in a space of dimension $t+1$, this is a contradiction.
\end{proof}

\subsection{Proof sketch} \label{subsec:sketch}
    Let us describe the main ideas behind the proof of Theorem~\ref{thm:main}. We first discuss the proof of the weaker bound $r(s, k) \ge k^{s-2+o(1)}$, where it is easier to see the main ideas and which motivates the proof of Theorem~\ref{thm:main}. We also remark that this construction is behind the proofs of all of our other results.
    
    Let $t = s-2$ and consider the graph $G = G(t, q)$ and recall that it is a $K_s$-free $(n, d, \lambda)$-graph with $n \sim q^{s-2}, d \sim n/ q \sim q^{s-3} \sim n^{1 - 1 / (s-2)}$ and $\lambda \sim \sqrt{d}$. Lemma~\ref{lem:Hs-free} tells us that $G$ also avoids a certain bipartite structure. This motivates a kind of product construction to obtain a $K_s$-free graph $\Gamma$ with roughly $n' = n^2$ vertices and degree $d' \sim n' / q$, which we describe shortly. Crucially, $\Gamma$ is sufficiently pseudorandom so that we can use the container method to count the number of independent sets of size $k = \Theta(n' (\log n')^2 / d')$. A standard sampling and deletion argument then yields the lower bound $r(s, k) \ge k^{s-2+o(1)}.$ We note that plugging in the graph $G$ into the argument of Mubayi and Verstraete~\cite{mubayi-verstraete} yields the lower bound $r(s, k) \ge k^{(s-1)/2 + o(1)}$. At a high level, the fact that the number of vertices goes from $n$ to $n^2$ in this product construction, corresponds to the improvement of about a factor of $2$ in the exponent in our lower bound for $r(s, k)$. We proceed with a more detailed outline of the proof, starting with the aforementioned product-like construction.

    The pseudorandom graph $\Gamma$ is constructed in two steps. First, define a directed graph $D$ with $V(D) = \{ (a, b) \mid ab \not\in E(G)\}$ and with a directed edge from $(a, b)$ to $(a', b')$ if and only if $(a, b') \in E(G)$. Observe that $D$ has $n' \sim n^2$ vertices and is $d'$-regular, where $d' \sim nd \sim n' / q$. Lemma~\ref{lem:Hs-free} implies that $D$ does not contain a copy of the transitive tournament on $s$ vertices. Now, take a random permutation $\pi$ of the vertices of $D$ and let $\Gamma$ be the undirected graph on the same vertex set obtained by keeping all edges going forward according to $\pi$. Since $D$ does not contain a copy of the transitive tournament of size $s$, it follows that $\Gamma$ is $K_s$-free with probability $1$. It remains to count the expected number of independent sets of size $k = \Theta(n' (\log n')^2 / d')$ in $\Gamma$. Note that an independent set of size $k$ in $\Gamma$ corresponds to a \emph{forward independent $k$-tuple} in $D$, defined as a $k$-tuple of vertices $(v_1, \dots, v_k) \in V(D)^k$ with no forward edges, i.e. such that $v_i v_j \not\in A(D),$ for any $i < j, i,j \in [k]$. Observe that each forward independent $k$-tuple in $D$ is increasing according to $\pi$ with probability $1/k!$. Thus, it suffices to bound the number forward independent $k$-tuples in $D$, or equivalently, $2k$-tuples $(a_1, b_1, \dots, a_k, b_k) \in V(G)^{2k}$ such that $a_i b_i \not\in E(G)$ for all $i \in [k]$ which we shall call \emph{coherent $2k$-tuples}. Indeed, note that for a forward independent $k$-tuple $((a_1, b_1), \dots, (a_k, b_k))$ it holds that $a_i b_i \not \in E(G)$, since $a_ib_i \in V(D)$, and for all $i < j,$ it holds that $a_i b_j \not\in E(G)$, since $((a_i, b_i), (a_j, b_j)) \not\in A(D)$. Thus, a forward independent $k$-tuple in $D$ corresponds to a coherent $2k$-tuple of vertices of $G$.

    Our argument for this follows that of Alon and R\"{o}dl~\cite{alon-rodl} for upper bounding the number of independent sets in $(n, d, \lambda)$-graphs. The expander mixing lemma shows that the largest independent set in an $(n, d, \lambda)$-graph has size at most $n \lambda / d$ and~\cite{alon-rodl} show that the number of independent sets of size $k = 100 n (\log n)^2 / d$, say, in an $(n, d, \lambda)$-graph is essentially upper bounded by the number of subsets of size $k$ of the largest possible independent set. In our setting, we show that the number of coherent $2k$-tuples is essentially upper bounded by the number of ways to choose a $k$-tuple of vertices from a set $A$ and a $k$-tuple of vertices from a set $B$, where there are no edges between $A$ and $B$ in $G$. The expander mixing lemma shows that $|A||B| \le \frac{\lambda^2 n^2}{d^2}$ if there are no edges between $A$ and $B$ which carries over to the fact that there are at most $O(\lambda^2 n^2 / d^2)^k$ coherent $2k$-tuples. Recalling that each of them yields an independent set in $\Gamma$ with probability $1/k! = \Theta(1/k)^k$, we see that the expected number of independent sets of size $k$ in $\Gamma$ is at most $O\left( \frac{\lambda^2 n^2}{d^2 k}\right)^k$. Hence, by sampling the vertices with probability $p = \Omega(\frac{d^2 k}{\lambda^2 n^2})$ and deleting any remaining independent $k$-set, we deduce that $r(s, k) \ge \Omega(p |V(D)|) = \Omega(d^2 k / \lambda^2) = k^{s-2+o(1)}$.
    
    The final modification to obtain Theorem~\ref{thm:main} is a slightly different product construction using the graphs $G(t, q)$, where, as before, we start by defining a directed graph. Crucially, we may take $t = s-1$ and ultimately produce a $K_s$-free graph, that is, we save one in the size of the largest clique compared to the approach described above. So, let $t=s-1, G = G(t, q)$ and let $D^* = D^*(t, q)$ be the digraph with $V(D^*) = \{ (a, b) \mid ab \in E(G)\}$ and a directed edge from $(a, b)$ to $(a', b')$ if and only if $ab' \in E(G),$ but $a' b \not\in E(G)$. Crucially, $D^*$ does not contain a copy of the transitive tournament of size $s$. See Figure~\ref{fig:configurations} for a comparison of the configurations in $G(t, q)$ that would yield a clique in the former and latter constructions. Denote $n = |V(G)| \sim q^{s-1}$ and observe that $D^*$ has $n_2 \sim n^2 / q \sim q^{2s-3}$ vertices and is $d_2$-regular, where $d_2 \sim n_2 / q \sim n_2^{1 - \frac{1}{2s-3}}$, which, unsurprisingly, matches the maximum possible density of an optimally pseudorandom $K_s$-free graph. As before, taking a random permutation of the vertices of $D^*$ and keeping all forward edges, we obtain a $K_s$-free graph $\Gamma^*$. While $\Gamma^*$ is not spectrally pseudorandom, we can count its independent sets, or rather, forward independent tuples in the digraph $D^*$ from which $\Gamma^*$ arises, using the container method. The analysis here is somewhat more involved but based on the same core ideas.
    
\subsection{Preliminaries} \label{subsec:prelims}
For a graph $G$ and a positive integer $k$, we denote by $i_k(G)$ the number of independent sets of size $k$ in $G$. 

We denote by $T_s$ the transitive tournament on $s$ vertices and say that a digraph is $T_s$-free if it does not contain a copy of $T_s$. For a digraph $D$, we say that a $k$-tuple of vertices $(v_1, v_2, \dots, v_k) \in V(D)^k$ is \emph{forward independent} if there are no $i, j \in [k], i < j$ such that $(v_i, v_j) \in A(D)$. We denote the number of forward independent $k$-tuples in a digraph $D$ by $\fwi{k}(D)$.

We start with two simple lemmas.
\begin{lemma} \label{lem:Ts-free-Ks-free}
    Suppose that $D$ is a $T_s$-free digraph on $n$ vertices. Then, for any $k \ge 1$, there exists a $K_s$-free graph $G$ on $n$ vertices satisfying $i_k(G) \le \left(\frac{e}{k} \right)^k \cdot \fwi{k}(D)$.
\end{lemma}
\begin{proof}
    Let $\pi$ be a random permutation of $V(D)$ and let $G$ be the graph with vertex set $V(D)$ and edge set $\{ \{u, v\} \mid (u, v) \in A(D), \pi(u) < \pi(v) \}$, that is, $G$ contains an edge for each arc going forward according to $\pi$. Since $D$ is $T_s$-free, it follows that $G$ is $K_s$-free. Observe that every independent set of size $k$ corresponds to a distinct forward independent $k$-tuple in $D$, each of which is in an increasing order according to $\pi$ with probability $1/k!$. By linearity of expectation, $\E[i_k(G)] = \fwi{k}(D) / k! \le \left( \frac{e}{k}\right)^k \cdot \fwi{k}(D)$, implying the existence of the desired graph.
\end{proof}

The following lemma is well-known and implicitly proved in \cite{mubayi-verstraete}.
\begin{lemma} \label{lem:sampling-Ks-free}
    Suppose $G$ is a $K_s$-free graph on $n$ vertices and let $k \ge 1$ be arbitrary. If $p \in [0, 1]$ satisfies $p^{k} i_k(G) \le 1$, then, $r(s, k) > pn - 1.$
\end{lemma}
\begin{proof}
    Let $U \subseteq V(G)$ be a random set of vertices obtained by keeping each vertex independently with probability $p$ and let $G' = G[U]$ be the induced subgraph on $U$. Clearly, $G'$ is $K_s$-free and $\E[ |V(G')| - i_k(G')] = p n - p^k i_k(G) \ge pn - 1$. Thus, there exists a $K_s$-free graph $F$ with $|V(F)| - i_k(F) \ge pn - 1$. Removing an arbitrary vertex from each independent set of size $k$ in $F,$ we obtain a $K_s$-free graph with no independent set of size $k$ and at least $pn - 1$ vertices. Hence, $r(s, k) > pn - 1$, as claimed.
\end{proof}

We shall use the well-known expander mixing lemma. Though usually stated for a graph without loops, the same proof works if loops are allowed.
\begin{lemma}[e.g.~{\cite[Corollary~9.2.5]{alon-spencer}}] \label{lem:expander-mixing}
    Let $G$ be an $(n, d, \lambda)$-graph and let $A, B \subseteq V(G)$. Then,
    \[ \big| e_G(A, B) - \frac{d}{n} |A||B| \big| \le \lambda \sqrt{|A||B|} . \]
\end{lemma}

A useful corollary of the expander mixing lemma is the following.
\begin{lemma}[{\cite[Lemma~2.2]{alon-rodl}}] \label{lem:BC}
    Let $G$ be an $(n, d, \lambda)$-graph and let $B \subseteq V(G)$ be an arbitrary subset of vertices of $G$. Define
    \[ A = \left\{ u \in V \mid |N_G(u) \cap B| \le \frac{d|B|}{2n} \right\}. \]
    Then,
    \[ |A| |B| \le \frac{4 \lambda^2}{d^2} n^2. \]
\end{lemma}

The following simple combinatorial lemma codifies our strategy for counting forward independent tuples in a digraph. We remark that in our application of Lemma~\ref{lem:tree-counting} the rooted tree $\cT$ will be the tree formed by forward independent tuples with the root being the empty tuple and the children of a forward independent $m$-tuple being the forward independent $(m+1)$-tuples obtained by appending a vertex.
\begin{lemma} \label{lem:tree-counting}
    Let $\cT$ be a (not necessarily finite) rooted tree such that every vertex has at most $\Delta$ children out of which at most $h$ are \emph{marked}. Furthermore, suppose that every path from the root has at most $w$ unmarked vertices. Then, for any $k \ge w$, the number of paths of length $k$ from the root is at most 
    \[ 2^k \Delta^w h^{k-w}. \] 
\end{lemma}
\begin{proof}
    Let $v_0$ denote the root of $\cT$. For a path $P = v_0, \dots, v_m$ from the root, define its signature by $\phi(P) = (z_1, \dots, z_m) \in \{0, 1\}^m$, as the binary string with $z_i = 0$ if $v_i$ is marked, and $z_0 = 1$, otherwise.
    
    By induction on $m$, we prove that for any binary string $\bz = (z_1, \dots, z_m) \in \{0, 1\}^m$, the number of paths with signature $\bz$ is at most $\Delta^{|\bz|} \cdot h^{m - |\bz|}$, where we denote $|\bz| = \sum_{i=1}^m z_i$.
    
    The base case $m = 0$ is trivial. Now let $m \ge 1$, consider an arbitrary string $\bz = (z_1, \dots, z_m) \in \{0, 1\}^m$ and denote $\bz' = (z_1, \dots, z_{m-1})$. Note that for any path $v_0, v_1, \dots, v_m$ in $\cT$ such that $\phi(v_0, v_1, \dots, v_m) = \bz$, it holds that $\phi(v_0, v_1, \dots, v_{m-1}) = \bz'$.

    First, assume that $z_m = 0$. Consider a path $P' = v_0, v_1,  \dots, v_{m-1}$ in $\cT$ with $\phi(P') = \bz'$. By definition, if $P = v_0, v_1, \dots, v_m$ satisfies $\phi(P) = \bz$, then $v_m$ is a marked child of $v_{m-1}$. Since each vertex has at most $h$ marked children, by the induction hypothesis, it follows that the number of paths from the root with signature $\bz$ is at most $( \Delta^{|\bz'|} \cdot h^{m-1 - |\bz'|} ) \cdot h = \Delta^{|\bz|} h^{m - |\bz|}$, where we used that $|\bz| = |\bz'|$.
    
    Now, suppose that $z_m=1$. For any path $v_0, v_1, \dots, v_{m-1}$ with signature $\bz'$, there are trivially at most $\Delta$ choices for a child $v_m$ of $v_{m-1}$ such that $\bz$ is the signature of the path $v_1, \dots, v_m$. Thus, by the induction hypothesis, the number of paths from the root with signature $\bz$ is at most $(\Delta^{|\bz'|} h^{m-1-|\bz'|}) \cdot \Delta = \Delta^{|\bz|} h^{m - |\bz|},$ where we used that $|\bz| = |\bz'| + 1$. This concludes the inductive step.

    Hence, for any $k \ge w$, the number of vertices in $\cT$ at distance $k$ from the root is at most 
    \[ \sum_{\bz \in \{0,1\}^k \, \colon \, |\bz| \le w} \Delta^{|\bz|} h^{k - |\bz|} \le 2^k \Delta^w h^{k-w}, \]
    where we used that $\Delta \ge h$.
\end{proof}

\subsection{A weaker bound and the proof of Theorem~\ref{thm:off-diagonal-general}} \label{subsec:old}

\textbf{The construction.}
Let $t \ge 2$, let $q$ be a prime power and denote $G = G(t, q)$. Let $D = D(t, q)$ be the directed graph whose vertex set is the set of all ordered pairs of vertices not forming an edge in $G$, i.e. $V(D) = \{ (a, b) \in V(G)^2 \mid (a, b) \not\in E(G) \}$, and $D$ has an arc from $(a, b)$ to $(a', b')$ if and only if $ab' \in E(G)$. 

Crucially, $D$ is $T_{t+2}$-free.
\begin{lemma} \label{lem:t+2-free}
    For any $t \ge 2$ and $q$ a prime power, the digraph $D$ is $T_{t+2}$-free.
\end{lemma}
\begin{proof}
    Denote $G = G(t, q)$ and $D = D(t, q)$ and suppose that $D$ contains a copy of $T_{t+2}$ formed by vertices $v_1, \dots, v_{t+2}$ in this order. For $i \in [t+2]$, let $v_i = (a_i, b_i)$, where $a_i, b_i \in V(G)$. Since for $i \in [t+2]$, $v_i \in V(D)$, by definition, we have that $a_i b_i \not\in E(G)$. Furthermore, for $i, j \in [t+2], i < j$ since $v_i v_j \in A(D),$ by definition, it holds that $a_i b_j \in E(G)$. This directly contradicts Lemma~\ref{lem:Hs-free}.
\end{proof}

Throughout the rest of this subsection, we shall assume that $q$ is sufficiently large, i.e. larger than some absolute constant. Hence, by Lemma~\ref{lem:properties}, we have that $G$ is an $(n, d, \lambda)$-graph with $n \in [q^t / 2, 2q^t], d \in [q^{t-1} / 2, 2q^t]$ and $\lambda \le 2 \sqrt{d}$.
Next, we show that $D$ contains few forward independent tuples. As mentioned, our proof for this closely follows the argument of Alon and R\"{o}dl~\cite{alon-rodl}.

\begin{lemma} \label{lem:old-fw-indep-counting}
    Let $q$ be a sufficiently large prime power, let $t \ge 2$ and set $w = 32 t q \log q$. Then, for any $k \ge w$,
    \[ \fwi{k}(D) \le 8^k \left(\frac{\lambda^2}{d^2}\right)^{k-w}  n^{2k}. \]
\end{lemma}
To facilitate the reader's understanding, let us point out that to obtain the lower bound $r(s, k) \ge k^{s-2+o(1)}$ for fixed $s$, we shall apply Lemma~\ref{lem:old-fw-indep-counting} with $k = \Theta(q (\log q)^2)$ so that the above bound reduces to $\fwi{k}(D) \le O\left( \frac{\lambda^2 n^2}{d^2} \right)^k$. On the other hand, to prove Theorem~\ref{thm:off-diagonal-general}, we shall have $q \approx \frac{\delta}{200} \frac{k}{s} / \log\left( \frac{k}{s}\right)$ so that the above bound implies $\fwi{k}(D) \le \left(\left(\frac{\lambda^2}{d^2}\right)^{1 - \delta/5} n^2\right)^k$.

\begin{proof}[Proof of Lemma~\ref{lem:old-fw-indep-counting}]
    Note that a forward independent $m$-tuple in $D$ corresponds to a $2m$-tuple $(a_1, b_1, \dots, a_m, b_m) \in V(G)^{2m}$ such that $a_i b_i \not \in E(G)$ for all $i \in [m]$ and $a_i b_j \not\in E(G)$ for all $i < j \in [m]$. We call such a $2m$-tuple \emph{coherent}.

    It will be convenient to view the set of coherent $2m$-tuples as a rooted tree $\cT$, where the root of $\cT$ is the empty tuple and the children of a coherent $2m$-tuple $(a_1, b_1, \dots, a_m, b_m)$ are the coherent $2(m+1)$-tuples obtained by appending a pair, i.e. coherent $2(m+1)$-tuples of the form $(a_1, b_1, \dots, a_m, b_m, a_{m+1}, b_{m+1})$.

    With the aim of applying Lemma~\ref{lem:tree-counting}, we set $h = \frac{4 \lambda^2 n^2}{d^2}$ and $\Delta = n^2$ and recall that $w = 32 t q \log q$. Let $m \ge 0$ be arbitrary and let $\sigma = (a_1, b_1, \dots, a_m, b_m)$ be an arbitrary coherent $2m$-tuple. Define 
    \[ B^\sigma= V(G) \setminus \bigcup_{i=1}^m N_G(a_i) \text{ and } A^\sigma = \left\{ a \in V(G) \mid |N_G(a) \cap B^\sigma| \le \frac{d}{2n} |B^\sigma| \right\}. \]

    By Lemma~\ref{lem:BC}, we have 
    \begin{equation} \label{eq:A-times-B}
        |A^\sigma||B^\sigma| \le \frac{4 \lambda^2 n^2}{d^2} = h.
    \end{equation}
    Observe that any pair $(a, b) \in V(D)$ that can be appended to $\sigma$, i.e. such that $(a_1, b_1, \dots, a_m, b_m, a, b)$ is coherent, satisfies $b \in B^\sigma$. Indeed, otherwise, there is some $i \in [m]$ such that $a_i b \in E(G)$, which is not allowed by definition. We mark all children of $\sigma$ obtained by appending a pair $(a, b) \in A^\sigma \times B^\sigma$ and the remaining children of $\sigma$ are left unmarked. By~\eqref{eq:A-times-B}, we have marked at most $h$ children of $\sigma$. 

    Now, suppose that for some $m$ there is a coherent $m$-tuple  $\sigma = (a_1, b_1, \dots, a_m, b_m)$ such that there are more than $w$ unmarked vertices on the path in $\cT$ from the root to $\sigma$. Then,
    \[ |B^\sigma| \le n \cdot \left( 1 - \frac{d}{2n}\right)^{w-1} \le 2 q^t \cdot \exp\left(- \frac{d}{2n} \cdot 16tq \log q \right) \le 2 q^t \cdot \exp\left( - \frac{1}{8q} \cdot 16t q \log q\right) < 1. \]
    This implies that $B^\sigma = \emptyset$, but $b_m \in B^\sigma$, a contradiction.

    Clearly, every vertex of $\cT$ has at most $n^2 = \Delta$ children, thus, by Lemma~\ref{lem:tree-counting}, we have
    \[ \fwi{k}(D) \le 2^k \Delta^w h^{k-w} = 2^k n^{2w} \left(\frac{4 \lambda^2 n^2}{d^2} \right)^{k-w} \le 8^k \left( \frac{\lambda^2}{d^2}\right)^{k-w} n^{2k}, \]
    as needed.
\end{proof}

We are ready to deduce the weaker lower bound $r(s, k) \ge k^{s-2+o(1)}$.

\begin{theorem} \label{thm:s-2}
    For any $s \ge 4$ there is a constant $c_s > 0$, such that for all $k \ge 2,$ it holds that
    \[ r(s, k) \ge c_s \frac{k^{s-2}}{(\log k)^{2s-6}}. \]
\end{theorem}
\begin{proof}
    By reducing $c_s$, we may assume that $k$ is sufficiently large. Denote $t = s-2$ and let $q$ be a prime power, say a power of $2$, such that $100t^2 q (\log q)^2 \le k \le 400 t^2 q (\log q)^2,$ which clearly exists for any sufficiently large $k$. Let $D = D(t, q)$. By Lemma~\ref{lem:t+2-free}, $D$ is $T_s$-free and by Lemma~\ref{lem:old-fw-indep-counting}, denoting $w = 32t q \log q,$ we have
    \[ \fwi{k}(D) \le 8^k \left(\frac{\lambda^2}{d^2}\right)^{k - w} n^{2k} \le 8^k \left(\frac{4}{d}\right)^{k-w} n^{2k} \le \left( \frac{64 n^2}{d}\right)^k, \]
    where in the last inequality we used that $d^w \le 2^k$, since $k \ge 100 t^2 q (\log q)^2, w = 32 t q \log q$ and $d \le 2q^{t-1}$. 
    Applying Lemma~\ref{lem:Ts-free-Ks-free}, we obtain a $K_s$-free graph $\Gamma$ with $|V(D)|$ vertices such that 
    \[ i_k(\Gamma) \le \left( \frac{e}{k}\right)^k \fwi{k}(D) \le \left(\frac{200 n^2}{kd}\right)^k. \]
    Finally, set $p = \frac{kd}{200 n^2},$ so that $p^k i_k(\Gamma) \le 1$ and note that clearly $p \le 1$. Hence, Lemma~\ref{lem:sampling-Ks-free} yields
    \[ r(s, k) \ge p |V(\Gamma)| - 1 \ge p (n^2 / 2) - 1 \ge \frac{kd}{400} - 1 \ge \frac{k q^{s-3}}{800} - 1 \ge c_s \frac{k^{s-2}}{(\log k)^{2s-6}}, \]
    for a sufficiently small constant $c_s>0$, where in the last inequality we used that $q \ge c'_s k / (\log k)^2$ for a positive constant $c'_s$ depending only on $s$.
\end{proof}

The proof of Theorem~\ref{thm:off-diagonal-general} is nearly identical, except the optimal choice of $q$ is different.

\begin{proof}[Proof of Theorem~\ref{thm:off-diagonal-general}]
    Fix $\delta > 0$ and note that we may assume that $\delta < 1/10$, say. Let $L = L(\delta)$ be a sufficiently large constant to be chosen implicitly later. Let $s, k$ be given positive integers such that $s \ge L$ and $k \ge Ls$. Let $q$ be a prime power, say a power of $2$, such that $q \le \frac{\delta}{200} \cdot \frac{k}{s} / \log\left( \frac{k}{s} \right) \le 2q$. Note that we may assume that $q$ is sufficiently large since $k/s \ge L$. Also, observe that by taking $L$ sufficiently large, we have $k/s \ge q \ge \left( \frac{k}{s} \right)^{1 - \delta / 2},$ since $k/s \ge L$.
    
    Let $t = s-2, w = 32 t q \log q$ and note that 
    \[ w \le 32t \cdot  \left(\frac{\delta}{200} \cdot \frac{k}{s} / \log \left( \frac{k}{s}\right)\right) \cdot \log\left( \frac{k}{s} \right) \le  \frac{\delta}{5} \cdot k.\]
    Let $D = D(t, q)$. By Lemma~\ref{lem:t+2-free}, $D$ is $T_s$-free and by Lemma~\ref{lem:old-fw-indep-counting},
    \[ \fwi{k}(D) \le 8^k \left(\frac{\lambda^2}{d^2}\right)^{k-w} n^{2k} \le 32^k d^{-(k-w)} n^{2k} \le \left(32 d^{-(1 - \delta/5)} n^2 \right)^k. \]
    Applying Lemma~\ref{lem:Ts-free-Ks-free}, we obtain a $K_s$-free graph with $|V(D)|$ vertices and 
    \[ i_k(\Gamma) \le \left(\frac{e}{k}\right)^k \fwi{k}(D) \le \left(d^{-(1 - \delta/5)} n^2 \right)^k, \]
    where we used that $k$ is sufficiently large. Set $p = d^{1 - \delta/5} / n^2$ and note that $p \le 1$ and $p^k i_k(\Gamma) \le 1.$ Finally, applying Lemma~\ref{lem:sampling-Ks-free}, we obtain
    \[ r(s, k) \ge p |V(\Gamma)| - 1 \ge d^{1 - \delta/5} / n^2 \cdot (n^2 / 2) - 1 \ge d^{1 - \delta/5} / 4 \ge (q^{t-1} / 2)^{1 - \delta/5} / 4 \ge \left( \frac{k}{s}\right)^{(1 - \delta) s}, \]
    where in the last inequality we used $t=s-2, q \ge (k/s)^{1-\delta/2}$ and $k/s \ge L,$ where $L$ is large compared to $\delta$.
\end{proof}

\subsection{Proof of Theorem~\ref{thm:main}} \label{subsec:s-1}
In this subsection, we prove Theorem~\ref{thm:main}. The key observation that improves the result of the previous subsection is the following. To obtain the lower bound $r(s, k) \ge k^{s-2 + o(1)}$, we used that the graph $G = G(t, q)$ has no vertices $a_1, b_1, \dots, a_{t+2}, b_{t+2}$ such that $a_i b_i \not\in E(G)$ for all $i \in [t+2]$ and $a_i b_j \in E(G)$ for all $i < j, i, j \in [t+2]$. Clearly, this implies that $G$ also does not contain the configuration obtained by removing the vertices $b_1$ and $a_{t+2}$; see Figure~\ref{fig:configurations} for an illustration. We construct a digraph $D^*(t, q)$ designed to exploit this fact as follows.

\begin{figure}
    \centering
    \includegraphics[width=0.9\linewidth]{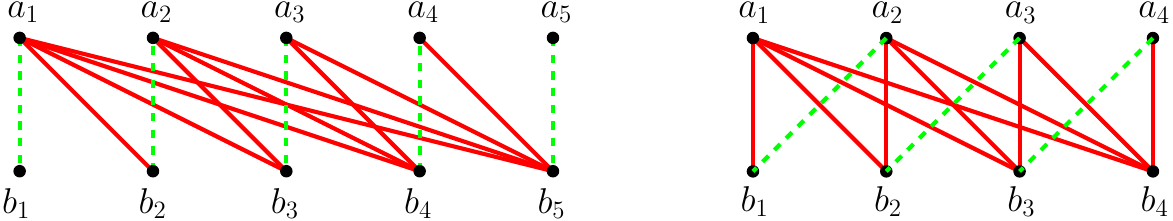}
    \caption{A representation of the configurations avoided in $G = G(t, q)$ with $t=3$. The full red lines represent edges of $G$, while the dashed green lines represent nonedges. The forbidden configuration on the left was used to obtain the lower bound $r(s, k) \ge k^{s-2+o(1)}$, and the one on the right will be used to obtain the tight bound $r(s, k) \ge k^{s-1+o(1)}$. Note that the latter configuration is obtained from the former by removing the vertices $b_1$ and $a_{t+2}$ and renaming $b_{i+1}$ by $b_i$ for $i \in [t+1]$.}
    \label{fig:configurations}
\end{figure}

\textbf{The construction.} Let $t \ge 2$, let $q$ be a prime power and denote $G = G(t, q)$. Let $D^* = D^*(t, q)$ be the directed graph with $V(D^*) = \{ (a, b) \mid ab \in E(G) \}$ and a directed edge from $(a, b)$ to $(a', b')$ if and only if $ab' \in E(G)$ and $a' b \not\in E(G)$. 

Crucially, we save one in the size of the largest transitive subtournament in $D^*$ compared to $D(t, q)$ used in Subsection~\ref{subsec:old}. As mentioned, this can be easily deduced from Lemma~\ref{lem:Hs-free}, but we also provide a direct proof.
\begin{lemma} \label{lem:Ts1-free}
    For every $t \ge 2$ and $q$ a prime power, $D^*(t, q)$ is $T_{t+1}$-free.
\end{lemma}
\begin{proof}
    Let $D^* = D^*(t, q)$ and for the sake of contradiction, suppose there are vertices $(a_1, b_1), (a_2, b_2), \dots, (a_{t+1}, b_{t+1})$ forming a copy of $T_{t+1}$ in $D^*$ in this order. By definition, there are corresponding nonzero vectors $\bx_1, \by_1, \dots, \bx_{t+1}, \by_{t+1} \in \bF_q^{t+1}$ such that for all $i \in [t+1],$ we have $\langle \bx_i, \by_i \rangle = 0$, since $(a_i, b_i) \in V(D^*)$, and for all $i < j, i, j \in [t+1]$, it holds that $\langle \bx_i, \by_j \rangle = 0$ and $\langle \bx_j, \by_i \rangle \neq 0$, since $((a_i, b_i), (a_j, b_j)) \in A(D^*)$. We remark that the only pairs for which we shall use that $\langle \bx_j, \by_i \rangle \neq 0$ are $(i+1, i), i \in [t]$. Since $\by_{t+1}$ is a nonzero vector, we find a vector $\bx_{t+2} \in \bF_q^{t+1} \setminus \{0\}$ such that $\langle \bx_{t+2}, \by_{t+1} \rangle \neq 0.$

    We claim that the vectors $\by_1, \dots, \by_{t+1}$ are linearly independent. Indeed, suppose there is a nontrivial linear combination $\sum_{j=1}^{t+1} \alpha_j \by_j = 0$ and let $i$ be the minimum index $j \in [t+1]$ such that $\alpha_j \neq 0.$ Then,
    \[ 0 = \sum_{j=1}^{t+1} \alpha_j \by_j = \sum_{j=i}^{t+1} \alpha_j \by_j. \]

    Taking the inner product with $\bx_{i+1}$ and recalling that $\langle \bx_{i+1}, \by_j \rangle = 0$ for all $j$ satisfying $i+1 \le j,$ we have
    \[ \alpha_i \langle \bx_{i+1}, \by_i \rangle = 0. \]
    On the other hand, $\langle \bx_{i+1}, \by_i \rangle \neq 0$, which implies $\alpha_i = 0$, contradicting our choice of $i$.    
\end{proof}

The argument to count forward independent tuples in $D^*$ is slightly more involved than the one used for $D(t, q)$ in the previous subsection. The basic idea, however, is the same. We pick a forward independent tuple one vertex at a time. At each step, there are few ``cheap'' choices, while the other choices drastically reduce the number of available choices for future vertices.

\begin{lemma} \label{lem:fw-ind-sets-s1}
    For any $t \ge 2,$ there is a constant $C = C(t)$ such that for any $q \ge C$, the following holds. If $k \ge C q (\log q)^2$, then
    \[ \fwi{k}(D^*(t, q)) \le (C q^t)^k. \]
\end{lemma}
\begin{proof}
    Let $t \ge 2$ be fixed and let $C = C(t)$ be a sufficiently large constant to be chosen implicitly later. Consider an arbitrary prime power $q \ge C$ and denote $G = G(t, q)$ and $D^* = D^*(t, q)$. Choosing $C$ sufficiently large, by Lemma~\ref{lem:properties}, we have that $G$ is an $(n, d, \lambda)$ graph with $n \in [q^t/2, 2q^t]$, $d \in [q^{t-1}/2, 2q^{t-1}]$ and $\lambda \le 2 \sqrt{d}$.
    
    As before, consider the rooted tree $\cT$, where each vertex is a forward independent tuple in $D^*$, with the root being the empty tuple and where the children of a forward independent $m$-tuple are the forward independent $(m+1)$-tuples obtained by appending a vertex of $D^*$. We prove the following key claim.
    
    \begin{claim}
        There is a constant $A = A(t)$ depending only on $t$, such that for $h = A q^t$ and $w = A q \log q$, it is possible to mark some vertices of $\cT$ such that every vertex of $\cT$ has at most $h$ marked children and every path in $\cT$ starting from the root contains at most $w$ unmarked vertices.
    \end{claim}
    \begin{proof}[Proof of Claim]   
        For convenience, we mark the root of $\cT$. Recall that the vertices of $G$ are the $1$-spaces of $\bF_q^{t+1}$. We call a $2m$-tuple $(a_1, b_1, \dots, a_m, b_m)$ of $1$-spaces of $\bF_q^{t+1}$ \emph{consistent} if $a_i \perp b_i$ for all $i \in [m]$, and for all $i < j, i, j \in [m]$, if $a_i \perp b_j,$ then also $a_j \perp b_i$. Note that a forward independent $m$-tuple in $D^*$ corresponds to a consistent $2m$-tuple. Indeed, $(a_i, b_i) \in V(D^*) \iff a_i \perp b_i$ and $((a_i, b_i), (a_j, b_j)) \not\in A(D^*) \iff a_i \not\perp b_j \text{ or } a_j \perp b_i.$
        
        We now consider a given consistent $2m$-tuple $\sigma = (a_1, b_1, \dots, a_m, b_m)$ and describe which of its children in $\cT$ are to be marked. For each $1$-space $y \in V(G)$ let
        \[  W(y) = W^\sigma(y) = \bigplus\{ b_i \mid a_i \perp y \} \, \text{ and }  \, r(y) = r^\sigma(y) = \dim W^\sigma(y), \]
        where we use $+$ to denote the subspace sum.

        Consider a pair of $1$-spaces $(a, b)$ that extends $\sigma$, i.e. such that $(a_1, b_1, \dots, a_m, b_m, a, b)$ is consistent. Clearly, we must have $a \perp b$. Furthermore, for any $i \in [m]$, if $a_i \perp b$, then, by consistency, we have that $a \perp b_i$. In other words, $a \perp W(b)$. Hence, for every $b \in V(G)$, the number of $a$ such that $(a, b)$ extends $\sigma$ is at most
        \begin{equation} \label{eq:extensions-given-b}
         (q^{t+1-r(b)} - 1) / (q-1) \le 2q^{t - r(b)}.
        \end{equation}
        In particular, if $r(b) = t+1,$ there is no $a$ such that $(a, b)$ extends $\sigma$.

        For $r \in [0, t+1]$, let $U_r = U_r^\sigma = \{ y \in V(G) \mid r(y) \le r \}$ and let $Z_r = Z_r^\sigma = U_r \setminus U_{r-1}$, where we denote $U_{-1} = \emptyset$. 
        
        Now, consider $r \in [0, t]$ such that $|U_r| \ge 1$. Let $\ell = \ell_r^{\sigma} \in [0, r]$ be such that $|Z_\ell|$ is maximal among $|Z_0|, |Z_1|, \dots, |Z_r|$. Note that this choice of $\ell$ implies $|Z_{\ell}| \ge |U_r| / (r+1) \ge |U_r| / (2t).$ We call $u \in Z_r$ \emph{popular} (with respect to $\sigma$ and $r$) if $u$ is contained in at least $|Z_\ell| / (16q)$ subspaces $W(y)$ with $y \in Z_\ell$. Note that for all $y \in Z_\ell$, we have $\dim W(y) = \ell$, so $W(y)$ contains $(q^{\ell} - 1) / (q-1) \le 2q^{\ell-1}$ $1$-spaces. Hence, by double counting pairs $(u, y)$ with $u \subseteq W(y), u \in Z_r, y \in Z_\ell$, we observe that the number of popular $1$-spaces is at most 
        \[ \frac{|Z_\ell| \cdot (2q^{\ell-1})}{|Z_\ell| / (16q)} = 32 q^{\ell}. \]
        We mark all the children of $\sigma$ extended by a pair $(a, b)$ with $b$ popular with respect to $\sigma$ and $r$, where $r = r^\sigma(b)$. By~\eqref{eq:extensions-given-b}, summing over all $r \in [0, t]$, it follows that this way we mark at most
        \[ \sum_{r=0}^t 32q^{\ell^\sigma_r} \cdot 2q^{t - r} \le \sum_{r=0}^t 64 q^t \le 128t q^t \]
        children of $\sigma$, where we used that $\ell_r^\sigma \le r$ for any $r \in [0, t]$.

        We call a $1$-space $a \in V(G)$ \emph{poor} (with respect to $\sigma$ and $r$) if there are at most $|Z_\ell| / (8q)$ vectors $y \in Z_\ell$ such that $a \perp y$, or equivalently, $a$ is poor if and only if $|N_G(a) \cap Z_\ell | \le |Z_\ell| / (8q)$, where $\ell = \ell_r^\sigma$. Denoting by $P_r$ the set of poor vertices with respect to $\sigma$ and $r$, by definition we have 
        \begin{equation} \label{eq:eG_P_Zell}
            e_G(P_r, Z_\ell) \le |P_r| |Z_\ell| / (8q).
        \end{equation}
        On the other hand, by Lemma~\ref{lem:expander-mixing}, we have that
        \[ e_G(P_r, Z_\ell) \ge \frac{d}{n} |P_r||Z_\ell| - \lambda \sqrt{ |P_r| |Z_\ell|} \ge \frac{|P_r||Z_\ell|}{4q} - 4q^{(t-1)/2} \sqrt{|P_r||Z_\ell|} . \]
        Combining the last two inequalities, we obtain $|P_r||Z_\ell| \le 1024 q^{t+1}$. Since $|Z_r| \le |Z_\ell|$, another application of Lemma~\ref{lem:expander-mixing} yields
        \begin{equation} \label{eq:edge-poor}
            e_G(P_r, Z_r) \le \frac{d}{n} |P_r||Z_r| + \lambda \sqrt{ |P_r| |Z_r|} \le \frac{4}{q} \cdot 1024 q^{t+1} + 4 q^{(t-1)/2} \cdot \sqrt{1024 q^{t+1}}  \le 5000 q^t.
        \end{equation}
        Mark all children of $\sigma$ extended by a pair $(a, b)$ such that $a$ is poor with respect to $\sigma$ and $r(b)$. Noting that $r(b) = r$ implies $b \in Z_r$, each such pair is an edge in $G$ between $P_r$ and $Z_r,$ for some $r \in [0,t]$. Thus, summing over all $r \in [0, t]$, by~\eqref{eq:edge-poor}, it follows that we mark at most $10^4 t q^t$ children of $\sigma$ this way. 
        
        We now record a crucial fact. Suppose that $(a, b)$ extends $\sigma$ and that the child $\sigma'$ obtained after extending $\sigma$ by $(a, b)$ is unmarked. Let $r = r(b)$ and recall that $r \in [0,t]$ and $b \in Z^\sigma_r \subseteq U^\sigma_r$. Let $\ell = \ell_r^\sigma$. Since $a$ is not poor with respect to $\sigma$ and $r$, it holds that $|N_G(a) \cap Z^\sigma_\ell| \ge |Z^\sigma_\ell| / (8q)$, and since $b$ is not popular with respect to $\sigma$ and $r$, it is contained in at most $|Z^\sigma_\ell| / (16q)$ subspaces $W^\sigma(y), y \in Z^\sigma_\ell$. Hence, there are at least $|Z^\sigma_\ell| / (16q)$ $1$-spaces $y$ such that $y \perp a$ and $b \not\subseteq W^\sigma(y)$. Observe that for such $y,$ it holds that $\dim W^{\sigma'}(y) = \dim W^{\sigma}(y) + 1 = \ell + 1$. Hence, $|U_\ell^{\sigma'}| \le |U_\ell^\sigma| - |Z_{\ell}^\sigma| / (16q)$. Noting that $|U_\ell^\sigma| \le |U_r^\sigma| \le 2t |Z_\ell^\sigma|,$ we have $|U_\ell^{\sigma'}| \le \left(1 - \frac{1}{32t q} \right) |U_\ell^{\sigma}|$. We denote $\psi(\sigma') = \ell$ to record that the corresponding set $U_\ell^{\sigma'}$ is significantly smaller than $U_\ell^\sigma$. For future reference, note that if $\psi(\sigma') = \ell$, then $|U^\sigma_\ell| \ge |Z^\sigma_\ell| \ge |U^\sigma_r| / (2t) > 0$.

        Clearly, taking $A = A(t)$ large enough, we have marked at most $h = A q^t$ children of every vertex in $\cT$. It remains to show that every path from the root contains at most $w$ unmarked vertices. Indeed, let $m \ge 0$ be arbitrary and consider a consistent $2m$-tuple $(a_1, b_1, \dots, a_m, b_m)$. For $i \in [0, m],$ denote $\sigma_i = (a_1, b_1, \dots, a_i, b_i)$. If there are more than $w$ unmarked vertices among $\sigma_1, \dots, \sigma_m$ (note that $\sigma_0$ is the root, so it is marked), then for some $\ell \in [0, t],$ there are at least $w / (t+1)$ indices $i \in [m]$ such that $\sigma_i$ is unmarked and $\psi(\sigma_i) = \ell$. However, note that $|U_\ell^{\sigma_i}|$ is non-increasing in $i$ and that $|U_{\ell}^{\sigma_i}| \le (1 - \frac{1}{32tq}) |U_{\ell}^{\sigma_{i-1}}|$ for every $i \in [m]$ such that $\sigma_i$ is unmarked and $\psi(\sigma_i) = \ell$. Recalling that $w = A q \log q$ and taking $A$ sufficiently large, we have $(1 - \frac{1}{32tq})^{w / (t+1) - 1} \le e^{-w / (100 t^2q) } < 1 / (4q^t)$. Let $i^*$ be the maximal index $i \in [m]$ such that $\sigma_i$ is unmarked and $\psi(\sigma_i) = \ell$. Since $|U_{\ell}^{\sigma_0}| = n \le 2q^t$, it follows that $|U_\ell^{\sigma_{i^*-1}}| \le n \cdot (1 - \frac{1}{32tq})^{w / (t+1) - 1} < 1$, so $U_\ell^{\sigma_{i^*-1}}$ is empty. However, as noted above, $\psi(\sigma_{i^*}) = \ell$ implies that $U_\ell^{\sigma_{i^*-1}}$ is non-empty, a contradiction.
    \end{proof}
    Let $A, h$ and $w$ be given by the above claim. Since a forward independent $k$-tuple corresponds to a path in $\cT$ of length $k$ from the root, we may apply Lemma~\ref{lem:tree-counting} with $\Delta = |V(D^*)| \le 4 q^{2t-1}, h$ and $w$ to conclude that
    \[ \fwi{k}(D^*) \le 2^k (4 q^{2t-1})^w h^{k-w} \le 2^k q^{2tw} h^k \le 4^k (A q^t)^k \le (C q^t)^k, \]
    where in the third inequality we used that $k \ge C q (\log q)^2$ so that $q^{2tw} \le 2^k$, and in the last inequality we used that $C \ge 4A$.
\end{proof}
\begin{proof}[Proof of Theorem~\ref{thm:main}]
    Let $s \ge 3$ be fixed and assume that $k$ is sufficiently large. Let $t = s-1$ and let $C = C(t)$ be the constant given by Lemma~\ref{lem:fw-ind-sets-s1}. Let $q$ be a prime power, say a power of $2$, such that $C q (\log q)^2 \le k \le 4 C q (\log q)^2$, which clearly exists for $k$ sufficiently large. Let $D^* = D^*(t, q)$. By Lemma~\ref{lem:Ts1-free}, $D^*$ is $T_s$-free and by Lemma~\ref{lem:fw-ind-sets-s1}, we have $\fwi{k}(D^*) \le (Cq^t)^k$. Applying Lemma~\ref{lem:Ts-free-Ks-free}, we obtain a $K_s$-free graph $\Gamma^*$ with $|V(\Gamma^*)| = |V(D^*)| \ge q^{2t-1} / 4$ and 
    \[ i_k(\Gamma^*) \le \left( \frac{e}{k} \right)^k \fwi{k}(D^*) \le \left(\frac{e}{k}\right)^k (C q^t)^k. \]
    Setting $p = \frac{k}{e} (C q^t)^{-1},$ we have $p^k i_k(\Gamma^*) \le 1$. Clearly $p \le 1,$ so by Lemma~\ref{lem:sampling-Ks-free}, we conclude that
    \[ r(s, k) \ge p |V(\Gamma^*)| - 1 \ge \frac{k}{e C q^t} \cdot q^{2t-1} / 4 - 1 \ge \frac{k q^{t-1}}{20 C} \ge c_s \frac{k^{s-1}}{(\log k)^{2s-4},} \]
    for some constant $c_s > 0$ depending only on $s$, where in the last inequality we used that $q \ge c_s' k / (\log k)^2$ for a positive constant $c_s'$ depending only on $s$.
\end{proof}

We remark that for $s=3$, the $K_3$-free graph $\Gamma^*$ obtained in the above proof is very similar to a construction used by Codenotti, Pudl\'{a}k and Resta~\cite{codenotti-pudlak-resta} to provide a constructive lower bound for $r(3, k)$.

\section{Improved lower bounds close to the diagonal} \label{sec:close}
In this section, we improve the lower bounds for Ramsey numbers close to the diagonal as well as for diagonal multicolor Ramsey numbers. Roughly speaking, to obtain these results, we construct a $K_s$-free graph on about $4^s$ vertices in which the number of independent sets of size $s$ or a bit larger is not much greater than expected in a random graph. Perhaps unsurprisingly, we will use the field of size $q = 2$ and we use the construction given in Subsection~\ref{subsec:old}. Note that the digraphs $D^*(t, 2)$ constructed in Subsection~\ref{subsec:s-1} have density about $1/4$ so they are not suitable for the applications in this section.

As before, we start by constructing a $T_s$-free digraph with few forward independent tuples. We shall not be able to bound the number of such tuples using a spectral approach, but luckily, they can efficiently be counted directly. We remark that essentially the same argument was used by Conlon and Ferber~\cite{conlon-ferber} to count cliques in $G(t, q)$.
\begin{lemma} \label{lem:F2-fw-indep-sets}
    Let $s, k$ be given integers with $4 \le s \le k$. There exists a $T_s$-free directed graph $D$ without loops on $N = 2^{2s-3} - 2^{s-1} - 2^{s-2} + 1$ vertices such that 
    \[ \fwi{k}(D) \le \sum_{t=1}^{s-1} \binom{k}{t} 2^{(s-1)(t+k) - \binom{t+1}{2}}. \]
    In particular, if $k = s+a$, where $a \ge 0$ and $a = o(s)$, then
    \[ \fwi{k}(D) \le 2^{\frac{3}{2}s^2 + as - \frac{5}{2}s + o(s)}. \]
\end{lemma}
\begin{proof}
    Let $G = G(s-2, 2)$ and let $D = D(s-2, 2)$ be the directed graph constructed in Subsection~\ref{subsec:old}. We remind the reader that $V(D) = \{ (x, y) \mid xy \not\in E(G) \}$ and there is an arc from $(x, y)$ to $(x', y')$ if and only if $xy' \in E(G)$. Note that $D$ has no loops. By Lemma~\ref{lem:Hs-free}, $D$ is $T_s$-free. Denote $p = s-1$ and note that we may identify the vertex set of $G$ with $\mathbb{F}_2^p \setminus \{ \overrightarrow{0} \}$, where $\bx \by \in E(G) \iff \langle \bx, \by \rangle = 0.$ Observe that $|V(D)| = 2 |E(G)| = (2^p - 1) \cdot (2^{p-1} - 1) = 2^{2s-3} - 2^{s-1} - 2^{s-2} + 1$, as claimed.
    
    It remains to upper bound the number of forward independent $k$-tuples in $D$. Note that a forward independent $k$-tuple in $D$ corresponds to a $2k$-tuple $(\ba_1, \bb_1, \ba_2, \bb_2, \dots, \ba_k, \bb_k) \in (\mathbb{F}_2^p)^{2k}$ such that $\langle \ba_j, \bb_i \rangle = 1$ for all $i,j \in [k]$ with $j \le i.$ Let us call such a $2k$-tuple \emph{bad}.

    For a bad $2k$-tuple $(\ba_1, \bb_1, \dots, \ba_k, \bb_k) \in (\mathbb{F}_2^p)^{2k}$, we shall call the sequence $(r_0, r_1, \dots r_k)$ with $r_0 = 0$ and $r_i = \dim( \mathrm{span} \{ \ba_1, \dots, \ba_i \})$, for $i \in [k]$, the \emph{rank sequence} of this $2k$-tuple. We call a sequence $(r_0, r_1, \dots, r_k)$ \emph{valid} if it satisfies $r_0 = 0, r_1 = 1, r_i \in \{r_{i-1}, r_{i-1} + 1\}$, for all $i \in [k]$ and $r_k \le p$. Note that a sequence which is not valid cannot be a rank sequence of a $2k$-tuple in $(\mathbb{F}_2^p)^{2k}$.
    
    We count the number of bad $2k$-tuples as follows. Let us fix an arbitrary valid sequence $\br = (r_0, r_1, \dots, r_k)$. Consider a fixed index $i \in [k]$ and suppose we have already chosen $(\ba_1, \bb_1, \dots, \ba_{i-1}, \bb_{i-1})$ with rank sequence $(r_0, r_1, \dots, r_{i-1})$. If $r_i = r_{i-1}$, then $\ba_i \in \mathrm{span} \{ \ba_1, \dots, \ba_{i-1} \}$, so there are at most $2^{r_i}$ choices for $\ba_i$, and if $r_i = r_{i-1} + 1,$ we loosely bound the number of choices for $\ba_i$ by $2^p$. Furthermore, given $(\ba_1, \bb_1, \dots, \ba_{i-1}, \bb_{i-1}, \ba_i)$, there are at most $2^{p - r_i}$ choices for $\bb_i$. Indeed, by definition of a bad $2k$-tuple, we have $\langle \ba_j, \bb_i \rangle = 1$, for all $j, 1 \le j\le i$. Since $\dim (\mathrm{span} \{\ba_1, \dots, \ba_i \}) = r_i,$ there are either zero or $2^{p - r_i}$ choices for $\bb_i$ given $(\ba_1, \bb_1, \dots, \ba_{i-1}, \bb_{i-1}, \ba_i)$.

    Denote $t = r_k$. In total, the number of bad $2k$-tuples with rank sequence $\br$ is at most
    \begin{align*}
        \left( \prod_{i \in [k] \colon r_i = r_{i-1} + 1} 2^p \cdot 2^{p - r_i}\right) \left( \prod_{i \in [k] \colon r_i = r_{i-1}} 2^{r_i} \cdot 2^{p - r_i} \right) &= \left(\prod_{i=1}^t 2^{2p - i}\right) 2^{p(k - t)} = 2^{2p t - \binom{t + 1}{2} + pk - pt} = 2^{pt + pk - \binom{t+1}{2}}.
    \end{align*}
    Note that, for any $t \in [p],$ there are at most $\binom{k}{t}$ valid sequences $(r_0, r_1, \dots, r_k)$ with $r_k = t$. Hence, the total number of bad $2k$-tuples is at most
    \begin{align*}
        \sum_{t=1}^p \binom{k}{t} 2^{pt + pk - \binom{t+1}{2}} = \sum_{t=1}^{s-1} \binom{k}{t} 2^{(s-1)(t+k) - \binom{t+1}{2}},
    \end{align*}
    as needed.

    Now, denote $k = s+a$ and assume that $a \ge 0$ and $a = o(s)$. For $t \in [s-1],$ denote the $t$-th summand by $M_t = \binom{k}{t} 2^{(s-1)(t+k) - \binom{t+1}{2}}$. Writing $t = s - j$, we have $\binom{k}{t} = \binom{k}{k-t} = \binom{s+a}{j+a},$ thus for $j \in [s-1]$, we have 
    \[ M_{s-j} = \binom{s+a}{j+a} 2^{(s-1)(s-j+k) - \binom{s-j+1}{2}} \le \left(\frac{e(s+a)}{a+1}\right)^{j+a} \cdot 2^{s^2/2 + sk - 3s/2 - k + 3j/2 - j^2/2}. \]
    The right hand side is maximized at $j = \log_2\left( 
    \frac{e(s+a)}{a+1} \right) + \frac{3}{2},$ thus for all $t \in [s-1]$, we have
    \[ M_t \le \left(\frac{e(s+a)}{a+1}\right)^a \cdot 2^{s^2/2 + sk - 3s/2 - k + O(\log s)^2}. \]
    Since $a = o(s)$, we have $\left(\frac{e(s+a)}{a+1}\right)^a = 2^{o(s)}$. Hence, for all $t \in [s-1],$ it holds that
    \[ M_t \le 2^{s^2/2 + sk - 3s/2 - k + o(s)}. \]
    Summing over all $t$, we conclude that 
    \begin{equation*}
        \fwi{k}(D) \le s \cdot 2^{s^2/2 + sk - 3s/2 - k + o(s)} = 2^{\frac{3}{2}s^2 + as - \frac{5}{2}s + o(s)}. \qedhere
    \end{equation*}
\end{proof}
We remark that we could have been more careful when counting the number of choices for $a_i$, but this would not lead to a significant improvement of the final bounds. It might be interesting to note the similarity between the proofs of Lemma~\ref{lem:old-fw-indep-counting} and Lemma~\ref{lem:F2-fw-indep-sets}. Indeed, the unmarked vertices in the proof of the former correspond to the indices $i$ with $r_i = r_{i-1} + 1$ in the proof of the latter.

Next, we prove our claimed lower bound for Ramsey numbers extremely close to the diagonal.

\begin{proof}[Proof of Theorem~\ref{thm:close}]
    Denote $k = s+a$. By Lemma~\ref{lem:F2-fw-indep-sets}, there exists a $T_s$-free directed graph $D$ on $N = 2^{2s-3} - 2^{s-1} - 2^{s-2} + 1$ vertices satisfying 
    \[ \fwi{k}(D) \le 2^{\frac{3}{2}s^2 + as - \frac{5}{2}s + o(s)}. \]    
    Applying Lemma~\ref{lem:Ts-free-Ks-free}, we obtain a $K_s$-free graph $G$ on $N$ vertices with $i_k(G) \le \left( \frac{e}{k}\right)^k \fwi{k}(D)$. If $i_k(G) = 0, $ then $r(s, k) \ge N,$ which is clearly sufficient. Otherwise, applying Lemma~\ref{lem:sampling-Ks-free} with $p = i_k(G)^{-1/k}$, we get
    \begin{align*}
        r(s, s+a) &\ge N \cdot i_k(G)^{-1/k} - 1 \ge (1+o(1))2^{2s-3} \cdot \frac{s}{e} \cdot 2^{(-\frac{3}{2}s^2 - as + \frac{5}{2}s + o(s)) / (a+s)}\\ &= (1+o(1))2^{2s-3} \cdot \frac{s}{e} \cdot 2^{-\frac{3s}{2} + \frac{a}{2} + \frac{5}{2} - \frac{a^2 + 5a}{2(a+s)}}
        \ge (1 + o(1)) \frac{s}{e}\cdot 2^{(s + a - 1)/2 - a^2 / (2s)},
    \end{align*}
    as claimed.
\end{proof}

We remark that for the diagonal case $a = 0,$ Theorem~\ref{thm:close} recovers Erd\H{o}s' lower bound~\cite{erdos-lb} up to a $(1 + o(1))$ factor. Next, we derive our improved lower bounds for two color Ramsey numbers slightly off the diagonal. Observe that Theorem~\ref{thm:close} improves the lower bound~\eqref{eq:spencer-close-to-diag} on $r(s, s+a)$ given by Spencer's method for any $a \ge 5$ and $a = o(s)$. When $a$ is linear in $s$, Theorem~\ref{thm:k-Ck} yields a more convenient bound, which we prove next.

\begin{proof}[Proof~of~Theorem~\ref{thm:k-Ck}]
    Denote $k = Cs$. Let $D$ be the digraph given by Lemma~\ref{lem:F2-fw-indep-sets} and denote $N = |V(D)| = (1+o(1)) 2^{2s-3}$ and recall that $\fwi{k}(D) \le \sum_{t=1}^{s-1} \binom{k}{t} 2^{(s-1)(t+k) - \binom{t+1}{2}}$.

    Note that the quadratic expression $(s-1)(t+k) - \binom{t+1}{2}$ is maximized at $t = s - 3/2,$ where its value equals $(s-1)k + \frac{1}{2}(s - 3/2)^2$. Thus, we may crudely bound $\fwi{k}(D)$ by
    \[ \fwi{k}(D) \le 2^k \cdot 2^{(s-1)k + \frac{1}{2}(s - 3/2)^2} \le  2^{sk + s^2/2}, \]
    where we used that $s$ is sufficiently large.

    Applying Lemma~\ref{lem:Ts-free-Ks-free}, we obtain a $K_s$-free graph $G$ on $N$ vertices with $i_k(G) \le \left( \frac{e}{k}\right)^k \fwi{k}(D)$. If $i_k(G) = 0, $ then $r(s, k) \ge N,$ which is clearly sufficient. Otherwise, applying Lemma~\ref{lem:sampling-Ks-free} with $p = i_k(G)^{-1/k}$, yields
    \[ r(s, k) \ge N \cdot i_k(G)^{-1/k} - 1 \ge (1 + o(1)) 2^{2s-3} \cdot \frac{s}{e} \cdot 2^{(sk + s^2/2) \cdot (-1 / k)} - 1 \ge 2^{\left(1 - \frac{1}{2C} \right)s}, \]
    for large enough $s$.
\end{proof}

Lastly, using the random homomorphism approach of He and Wigderson~\cite{he-wigderson}, which traces its roots to the work of Alon and R\"{o}dl~\cite{alon-rodl}, we obtain an improvement on the diagonal multicolor Ramsey numbers.
\begin{proof}[Proof~of~Theorem~\ref{thm:multicolor}]
    Let $D$ be the $T_s$-free digraph with $N = (1+o(1)) 2^{2s-3}$ vertices and $\fwi{s}(D) \le 2^{\frac{3}{2}s^2 - \frac{5}{2}s + o(s)}$ given by Lemma~\ref{lem:F2-fw-indep-sets}. Let $n = 2^{(s/2 - 4)(\ell-1)}$. We will construct an $\ell$-coloring of the complete graph on the vertex set $[n]$ with no monochromatic clique of size $s$, which will imply the statement. The coloring is defined as follows. For each $c \in [\ell-1]$, let $\phi_c \colon [n] \rightarrow V(D)$ be a uniformly random mapping chosen independently for different $c \in [\ell-1]$. For any $i, j$ with $1 \le i < j \le n$, color the edge $ij$ with the minimum color $c \in [\ell-1]$ such that $(\phi_c(i), \phi_c(j)) \in A(D),$ and if no such color exists, color $ij$ with color $\ell$.

    Observe that there is no monochromatic copy of $K_s$ in any of the colors $c \in [\ell-1]$. Indeed, suppose that vertices $v_1, \dots, v_s$ with $1 \le v_1 < v_2 < \dots < v_s \le n$ form a monochromatic clique in color $c \in [\ell-1].$ By definition, it follows that for all $i, j \in [s], i<j$, we have that $(\phi_c(v_i), \phi_c(v_j)) \in A(D)$. Since $D$ has no loops, it follows that $\{\phi_c(v_1), \phi_c(v_2), \dots, \phi_c(v_s)\}$ forms a copy of $T_s$ (in this order) in $D$, a contradiction.

    For vertices $v_1, \dots, v_s$ with $1 \le v_1 < v_2 < \dots v_s \le n$, the set $\{v_1, \dots, v_s\}$ forms a monochromatic clique in color $\ell$ if and only if $(\phi_c(v_1), \phi_c(v_2), \dots, \phi_c(v_s))$ is a forward independent $s$-tuple in $D$, for each $c \in [\ell-1]$. Hence, the expected number of monochromatic cliques of size $s$ of color $\ell$ is at most
    \[ \binom{n}{s} \left( \frac{\fwi{s}(D)}{N^s} \right)^{\ell-1} < n^s \cdot \left( \frac{2^{3s^2 / 2}}{(2^{2s-4})^s} \right)^{\ell-1} = n^s \cdot 2^{(-s^2/2 + 4s)(\ell-1)} = 1. \]
    Hence, with positive probability, there is no monochromatic clique of size $s$ in color $\ell$, completing the proof.
\end{proof}

\section{Concluding remarks} \label{sec:concluding}
\begin{itemize}    
    \item \textbf{The Erd\H{o}s-Rogers function.} For fixed positive integers $p, s$ with $2 \le p < s,$ the so-called Erd\H{o}s-Rogers function, which we denote by $f_{p, s}(n)$, is the maximum value of $m$ such that any $K_s$-free graph on $n$ vertices contains an induced subgraph on $m$ vertices with no copy of $K_p$. Note that $f_{p, s}$ generalizes the Ramsey numbers since $f_{2, s}(n) \ge m$ if and only if $r(s, m) \le n$. For general $p$ and $s$, the best known upper bound is due to Krivelevich~\cite{krivelevich-erdos-rogers} who showed that $f_{p, s}(n) = O_{p,s}(n^{\frac{p}{s+1}} (\log n)^{\frac{1}{p-1}}).$ For $s \ge 2p,$ Theorem~\ref{thm:main} implies an improvement on the exponent for this function, namely it shows that $f_{p, s}(n) \le n^{\frac{p-1}{s-1} + o(1)}.$ Indeed, consider the $n$-vertex $K_s$-free graph with independence number at most $k = n^{\frac{1}{s-1} + o(1)}$ guaranteed by Theorem~\ref{thm:main} and suppose it contains a $K_p$-free induced subgraph on $t$ vertices. Then, using the upper bound in~\eqref{eq:old-bounds}, we have $t \le r(p, k) \le k^{p-1 + o(1)} = n^{\frac{p-1}{s-1} + o(1)},$ as claimed. I thank Oliver Janzer for pointing out this implication.
    
    \item \textbf{On pseudorandomness.} The $K_s$-free graph $\Gamma$ obtained in the proof of Theorem~\ref{thm:main} does not have sufficiently small eigenvalues to deduce a good bound on the number of independent sets directly from its spectrum. Indeed, let $U$ and $W$ be two mutually orthogonal subspaces of $\bF_q^s$ of dimension $s/2,$ where, for simplicity, we assume $s$ is even. Then, with high probability, $\Gamma$ contains an almost complete bipartite subgraph where one side contains vertices of the form $(u, x)$ with $u \in U$ and the other vertices of the form $(y, w)$ with $w \in W$ and the two parts have size $\Theta(q^{(s/2) - 1} \cdot q^{s-2}) = \Theta(q^{3s/2 - 3}).$ The number of edges inside the parts will be much smaller than the number of edges across, so this subgraph induces a graph with smallest eigenvalue at most $-\Theta(q^{3s/2 - 3})$ and by Cauchy's interlacing theorem, so does $\Gamma$. On the other hand, the average degree of $\Gamma$ is $\Theta(q^{2s-4})$.
\end{itemize}

\textbf{Acknowledgements:} I am grateful to the team at OpenAI for sharing the improvement which lead to the tight exponent in Theorem~\ref{thm:main} and allowing me to include it in the present paper. I would like to thank Yuval Wigderson for helpful discussions and a thorough reading of an earlier draft of this paper. I also acknowledge helpful discussions and comments from Zach Hunter, Oliver Janzer, Zhihan Jin, Sam Mattheus, Aleksa Milojevi\'{c}, Rob Morris, Benny Sudakov, Jacques Verstraete and Liana Yepremyan.

\bibliographystyle{plain}
\bibliography{references}

\appendix
\section{Spencer's lower bound close to the diagonal}
Let $s \rightarrow \infty$ and suppose that $a = o(s)$. Then, the best lower bound that can be achieved by Spencer's method~\cite{spencer77} using the Lov\'{a}sz local lemma is given by~\eqref{eq:spencer-close-to-diag}, which we restate for convenience.
\[ r(s, s+a) \ge (1 + o(1)) \frac{s}{e} \cdot 2^{(s+1)/2 + a/4 + O(a^2/s)}. \]
To simplify the exposition, we shall not actually prove this bound, but rather only show a better bound cannot be achieved with this method. We recall the asymmetric version of the local lemma.
\begin{lemma}[e.g.~\cite{alon-spencer}] \label{lem:lll}
    Let $A_1, \dots, A_m$ be events in an arbitrary probability space. A directed graph $D = (V, E)$ on the set of vertices $V = \{1, 2, \dots, m\}$ is called a dependency digraph for the events $A_1, \dots, A_m$ if for each $i, 1 \le i \le m$, the event $A_i$ is mutually independent of all the events $\{ A_j \mid (i,j) \not\in E\}.$ Suppose that $D$ is a dependency digraph for the above events and suppose there are real numbers $x_1, \dots, x_m$ such that $0 \le x_i < 1$ and $\Pr[A_i] \le x_i \prod_{(i,j) \in E}(1 - x_j)$, for all $1 \le i \le m$. Then, with positive probability, none of the events $A_i, i \in [m]$ hold.
\end{lemma}

Following Spencer, let $G$ be the binomial random graph on $n$ vertices with edge probability $p$, where we shall optimize for $p$ later. Denote $k = s+a$. For each set $S$ of $s$ vertices, let $A_S$ denote the event that $S$ forms a clique and, analogously, for each set $T$ of $k$ vertices, let $B_T$ denote the event that $T$ forms an independent set in $G$.

Aiming to use Lemma~\ref{lem:lll}, we assign to each event of type $A_S$, the same value $x_A$ and to each event of type $B_T,$ the same value $x_B$. Note that $\Pr[A_S] = p^{\binom{s}{2}}$ and $\Pr[B_T] = (1 -p)^{\binom{k}{2}}$. For each $s$-set $S$, there are at least $\binom{s}{2} \binom{n-s}{s-2}$ events of type $A_{S'}$ on which $A_S$ depends, and analogously for each $k$-set $T$, there are at least $\binom{k}{2} \binom{n-k}{k-2}$ events of type $B_{T'}$ that $B_T$ depends on. To apply~\ref{lem:lll}, in particular we require
\[ p^{\binom{s}{2}} \le x_A (1 - x_A)^{\binom{s}{2} \binom{n-s}{s-2}} \le x_A \exp\left(-x_A \binom{s}{2} \binom{n-s}{s-2}\right). \]
The right hand side is maximized by taking $x_A = \left(\binom{s}{2} \binom{n-s}{s-2}\right)^{-1}$, which implies
\[ p^{\binom{s}{2}} \le \left(\binom{s}{2} \binom{n-s}{s-2}\right)^{-1}. \]
Rearranging, we have
\[ (1+o(1)) \binom{s}{2} n^{s-2} / (s-2)! \le p^{-\binom{s}{2}}, \]
implying
\[ n \le (1 + o(1)) \frac{s}{e} \cdot (1/p)^{(s+1) / 2 + 1 / (s-2)}. \]
Completely analogously, we obtain
\[ n \le (1 + o(1)) \frac{s}{e} \cdot (1 / (1-p))^{(k+1)/2 + 1 /(k-2)}, \]
where we used that $k = (1+o(1))s$.

Together, these clearly imply $p \in [1/4, 3/4]$, so we may drop the terms $2 / (s-2)$ and $2/(k-2)$ from the previous equations to get
\begin{equation} \label{eq:lll-1}
    n \le (1 + o(1)) \frac{s}{e} \cdot (1/p)^{(s+1) / 2},
\end{equation} 
and,
\begin{equation} \label{eq:lll-2}
    n \le (1 + o(1)) \frac{s}{e} \cdot (1 / (1-p))^{(k+1)/2}.
\end{equation}

To maximize $n$, we solve for $p$ the equation
\[ \frac{s}{e} \cdot (1/p)^{(s+1) / 2} = \frac{s}{e} \cdot (1 / (1-p))^{(k+1)/2}. \]
Writing $p = \frac{1}{2}(1-\delta)$, the above reduces to
\[ \left( \frac{1+\delta}{1-\delta}\right)^{s+1} (1 + \delta)^{a} = 2^a. \]
The left hand side is at least $(1 + 2\delta)^{s} \ge e^{(2\delta - 4\delta^2) s},$ so $e^{(2\delta - 4\delta^2) s} \le 2^a,$ which, by monotonicity of the right hand side of \eqref{eq:lll-2} in $\delta$ and using that $a = o(s)$, implies that $\delta \le \frac{a \log 2}{2s} + O(a^2/s^2)$. Plugging back into \eqref{eq:lll-1}, finally yields
\[ n \le (1+o(1)) \frac{s}{e} \cdot \left( \frac{2}{1-\delta} \right)^{\frac{s+1}{2}} 
\le (1+o(1)) \frac{s}{e} \cdot 2^{\frac{s+1}{2}} \cdot (1 + \delta + 2\delta^2)^{\frac{s+1}{2}} \le (1+o(1)) \frac{s}{e} \cdot 2^{\frac{s+1}{2} + \frac{a}{4} + O\left( \frac{a^2}{s}\right)}. \]

\end{document}